\documentclass[preprint,12pt]{elsarticle}
\usepackage{amsmath}
\usepackage{amsthm}
\usepackage{amssymb,epsfig}
\usepackage[cp 1250]{inputenc}
\usepackage{a4}
\usepackage{graphics}
\usepackage[dvips]{color}

\newtheorem{lemma}{Lemma}[section]

\newtheorem{thm}[lemma]{Theorem}
\theoremstyle{definition}
\newtheorem{definition}[lemma]{Definition}
\newtheorem{remark}[lemma]{Remark}
\newtheorem{example}[lemma]{Example}

\newtheorem{proposition}[lemma]{Proposition}
\newtheorem{corollary}[lemma]{Corollary}

\def\Z{\mathbb Z}
\def\N{\mathbb N}
\def\Q{\mathbb Q}
\def\A{\mathcal A}
\def\B{\mathcal B}
\def\C{\mathcal C}
\def\pf{\begin{proof}}
\def\pfk{\end{proof}}
\def\le{\leqslant}
\def\leq{\leqslant}
\def\ge{\geqslant}
\def\geq{\geqslant}

\begin{document}

\begin{frontmatter}

\title{$k$-block parallel addition versus $1$-block parallel addition\\ in non-standard numeration systems}

\author[liafa]{Christiane~\textsc{Frougny}}
\author[doppler]{Pavel~\textsc{Heller}}
\author[doppler]{Edita~\textsc{Pelantov\'a}\corref{cor1}}
\author[doppler]{Milena~\textsc{Svobodov\'a}}

\cortext[cor1]{Corresponding author}
\address[liafa]{LIAFA, UMR 7089 CNRS \& Universit\'e Paris 7\\
Case 7014, 75205 Paris Cedex 13, France}
\address[doppler]{Doppler Institute for Mathematical Physics and Applied Mathematics,\\
and Department of Mathematics, FNSPE, Czech Technical University, \\
Trojanova 13, 120 00 Praha 2, Czech Republic}

%%%%%%%%%%%%%%%%%%%%%%%%%%%%%%%%%%%%%%%%%%%%%%%%%%%%%%%%%%%%%%%%%%%%%%%%%%%%%%%%%%%%%%%%%

\begin{abstract}

Parallel addition in integer base is used for speeding up
multiplication and division algorithms.  $k$-block parallel
addition has been introduced by Kornerup in \cite{Kornerup}:
instead of manipulating single digits, one works with blocks of
fixed length~$k$. The aim of this paper is to investigate how such
notion influences the relationship between the base and the
cardinality of the alphabet allowing parallel addition. In this
paper, we mainly focus on a certain class of real bases --- the
so-called Parry numbers. We give lower bounds on the cardinality
of alphabets of non-negative integer digits allowing block
parallel addition. By considering quadratic Pisot bases, we are
able to show that these bounds cannot be improved in general and
we give explicit parallel algorithms for addition in these cases.
We also consider the $d$-bonacci base, which satisfies the
equation $X^d = X^{d-1} + X^{d-2} + \cdots + X + 1$. If in a base
being a $d$-bonacci number $1$-block parallel addition is possible
on the alphabet~$\A$, then $\#\A \geq d+1$; on the other hand,
there exists a $k\in\N$ such that $k$-block parallel addition in
this base is possible on the alphabet $\{0,1,2\}$, which cannot be
reduced. In particular, addition in the Tribonacci base is
$14$-block parallel on alphabet $\{0,1,2\}$.

\end{abstract}

%%%%%%%%%%%%%%%%%%%%%%%%%%%%%%%%%%%%%%%%%%%%%%%%%%%%%%%%%%%%%%%%%%%%%%%%%%%%%%%%%%%%%%%%%


\end{frontmatter}

%%%%%%%%%%%%%%%%%%%%%%%%%%%%%%%%%%%%%%%%%%%%%%%%%%%%%%%%%%%%%%%%%%%%%%%%%%%%%%%%%%%%%%%%%
%%%%%%%%%%%%%%%%%%%%%%%%%%%%%%%%%%%%%%%%%%%%%%%%%%%%%%%%%%%%%%%%%%%%%%%%%%%%%%%%%%%%%%%%%

\section{Introduction}

%%%%%%%%%%%%%%%%%%%%%%%%%%%%%%%%%%%%%%%%%%%%%%%%%%%%%%%%%%%%%%%%%%%%%%%%%%%%%%%%%%%%%%%%%

This work is a continuation of our two papers~\cite{FrPeSv1} and
\cite{FrPeSv2} devoted to  the study of parallel addition. Suppose
that two numbers $x$ and $y$ are given by their expansion
$x=\bullet x_1x_2\cdots$ and $y=\bullet y_1y_2\cdots$ in a given
base~$\beta$, and the digits $x_j$'s and $y_j$'s are elements of a
digit set $\A$. A parallel algorithm to compute their sum
$z=x+y=\bullet z_1z_2\cdots$ with $z_j \in \A$ exists when the
digit $z_j$ can be determined by the examination of a window of
fixed length around the digit $(x_j+y_j)$. This avoids carry
propagation.

Parallel addition has received a lot of attention, because the
complexity of the addition of  two numbers becomes constant, and
so it is used for internal addition in multiplication and division
algorithms, see \cite{ErcegovacLang} for instance.

A parallel algorithm for addition has been given by
Avizienis~\cite{Avizienis} in 1961; there,  numbers are
represented in base $\beta = 10$ with digits from the set $\A =
\{-6,-5,\ldots,5,6\}$. This algorithm has been generalized to any
integer base $\beta \ge 3$. The case $\beta =2$ and alphabet $\A =
\{-1,0,1\}$ has been elaborated by Chow and
Robertson~\cite{ChowRobertson} in 1978. It is known that the
cardinality of an alphabet allowing parallel addition in integer
base $\beta \ge 2$ must be at least equal to $\beta +1$.

We consider non-standard numeration systems, where the base is a
real or complex number $\beta$ such  that $|\beta|>1$, and the
digit set $\A$ is a finite alphabet of contiguous integer digits
containing~$0$. If parallel addition in base~$\beta$ is possible
on $\A$, then $\beta$ must be an algebraic number.

In \cite{FrPeSv1}, we have shown that if $\beta$ is an algebraic
number, $|\beta |>1$, such that all  its conjugates in modulus
differ from~$1$, then there exists a digit set $\A\subset \Z$ such
that addition on $\A$ can be performed in parallel. The proof
gives a method for finding a suitable alphabet~$\A$ and provides
an algorithm --- a generalization of Avizienis' algorithm --- for
parallel addition on this alphabet. But the obtained digit set
$\A$ is in general quite large, so in \cite{FrPeSv2} we have given
lower bounds on the cardinality of minimal alphabets (of
contiguous integers containing~$0$) allowing parallel addition for
a given base~$\beta$.

In \cite{Kornerup}, Kornerup has proposed a more general concept
of parallel addition. Instead of  manipulating single digits, one
works with blocks of fixed length~$k$. So, in this terminology,
the ``classical" parallel addition is just $k$-block parallel
addition with $k=1$.

The aim of this article is to investigate how the Kornerup's
generalization influences the  relationship between the base and
the alphabet for parallel addition, in the hope of reducing the
size of the alphabet. For instance, consider the Penney numeration
system with the complex base $\beta = \imath -1$,
see~\cite{Penney}. We know from \cite{FrPeSv2} that $1$-block
parallel addition in base $\imath-1$ requires an alphabet of
cardinality at least $5$, whereas Herreros in \cite{Herreros}
gives an algorithm for $4$-block parallel addition on the alphabet
$\A = \{-1,0,1\}$.

\bigskip

The paper is organized as follows. Definitions and previous
results are  recalled in Section~\ref{Prel}. In
Section~\ref{necessary}, we show that for an algebraic base with a
conjugate of modulus~$1$, block parallel addition is never
possible, Theorem~\ref{conjugate1}.

Then we consider a simple Parry number~$\beta$ whose R\'enyi
expansion of  unity $d_\beta(1) = t_1t_2\cdots t_m$ is such that $1 \leq
t_m  \leq t_i$  for $1 \le i \le m$, and we show that if block
parallel addition in this base is possible on the alphabet $\A =
\{0,1,\ldots, M\}$, then $M\geq t_1+t_m$, Theorem~\ref{Simple}.

For a non-simple Parry number~$\beta$ with the R\'enyi expansion
of unity  of the form $d_\beta(1) = t_1 t_2 \cdots t_m (t_{m+1}
t_{m+2} \cdots t_{m+p})^\omega$, one proves that if block parallel
addition is possible in base~$\beta$ on alphabet $\A = \{0, 1,
\ldots, M\}$, then $M \geq 2t_1-t_2 -1$, provided that a certain
set of conditions is satisfied, as described in detail in
Theorem~\ref{nonSimple}.

By considering quadratic Pisot bases, we are able to show that the
two  previously mentioned (lower) bounds for Parry numbers cannot
be improved in general. We give explicit parallel algorithms for
addition in these two cases (simple quadratic Parry numbers, and
non-simple quadratic Parry numbers).

The main result of Section~\ref{kblockhelps} is
Theorem~\ref{Pavel}, which  implies that there are many bases for
which the Kornerup's concept of block parallel addition reduces
substantially the size of the alphabet.

A number $\beta>1$ is said to {\em satisfy the (PF) Property} if
the sum of  any two positive numbers with finite greedy
$\beta$-expansion in base~$\beta$ has its greedy $\beta$-expansion
finite as well. We deduce that if $\beta >1$ satisfies the (PF)
Property, then there exists a $k \in \N$ such that $k$-block
parallel addition is possible on the alphabet $\A =
\{0,1,\ldots,2\lfloor \beta \rfloor\}$.

We then consider a class of well studied Pisot numbers, that
generalize  the golden mean $\frac{1+\sqrt{5}}{2}$. Let  $d$ be in $\N$, $d\geq 2$. The
real root $\beta >1$ of the equation $X^d = X^{d-1}+X^{d-2}+\cdots
+ X +1$ is said to be the {\em $d$-bonacci number}. These numbers
satisfy the (PF) Property. If, in base a~$d$-bonacci number
$1$-block parallel addition is possible on the alphabet~$\A$, then
$\#\A \geq d+1$; moreover, there exists some $k \in \N$ such that
$k$-block parallel addition is possible on the alphabet $\A =
\{0,1,2\}$, and this alphabet cannot be further reduced. In
particular, addition in the Tribonacci base is $14$-block parallel
on $\A = \{0,1,2\}$.

Part of our results concerns only  non-negative alphabets. The
reason is  simple. For non-negative alphabet a strong tool ---
namely the greedy expansions of numbers --- can be applied when
proving theorems. That is why  we recall some properties of the
greedy expansions in Section~\ref{Renyi}.

%%%%%%%%%%%%%%%%%%%%%%%%%%%%%%%%%%%%%%%%%%%%%%%%%%%%%%%%%%%%%%%%%%%%%%%%%%%%%%%%%%%%%%%%%
%%%%%%%%%%%%%%%%%%%%%%%%%%%%%%%%%%%%%%%%%%%%%%%%%%%%%%%%%%%%%%%%%%%%%%%%%%%%%%%%%%%%%%%%%

\section{Preliminaries}\label{Prel}

%%%%%%%%%%%%%%%%%%%%%%%%%%%%%%%%%%%%%%%%%%%%%%%%%%%%%%%%%%%%%%%%%%%%%%%%%%%%%%%%%%%%%%%%%

\subsection{Numeration systems}\label{Renyi}

%%%%%%%%%%%%%%%%%%%%%%%%%%%%%%%%%%%%%%%%%%%%%%%%%%%%%%%%%%%%%%%%%%%%%%%%%%%%%%%%%%%%%%%%%

For a detailed presentation of these topics, the reader may consult~\cite{cant}.

A \textit{positional numeration system} $(\beta,\A)$ within the
complex field  $\mathbb{C}$ is defined by a \textit{base} $\beta$,
which is a complex number such that $|\beta|>1$, and a
\textit{digit set}  $\A$ usually called the \textit{alphabet},
which is a subset of $\mathbb{C}$. In what follows, $\A$ is finite
and contains~$0$. If a complex number $x$ can be expressed in the
form $ \sum_{-\infty \le j\leq n} x_j\beta^j$ with coefficients
$x_j$ in $\A$, we call the sequence $(x_j)_{-\infty \le j\leq n}$
a $(\beta, \A)$\textit{-representation} of $x$ and note
$x=x_nx_{n-1}\cdots x_0\bullet x_{-1}x_{-2} \cdots$. If a $(\beta,
\A)$\textit{-representation} of $x$ has only finitely many
non-zero entries, we say that it is {\em finite} and the trailing
zeroes are omitted.

In analogy with the classical algorithms for arithmetical
operations, we work  only on the set of numbers with \emph{finite}
representations, i.e., on the set
\begin{equation}\label{Fin_A_beta}
{\rm Fin}_{\A}(\beta) = \Bigl\{ \ \sum_{j\in I}\ {x_j\beta^j} \mid I \subset \Z,  \ \ I \ \ \hbox{finite}, \ \ x_j \in \A\Bigr\}.
\end{equation}
Such a finite sequence $(x_j)_{j\in I}$ of elements of $\A$ is
identified with a  bi-infinite string $(x_j)_{j\in \Z}$ in
$\A^\Z$, where only a finite number of digits $x_j$ have non-zero
values.

\bigskip

When the base is a real number, the domain has been extensively
studied. The  best-understood case is the one of representations
of real numbers in a base $\beta >1$, the so-called {\em greedy
expansions}, introduced by R\'enyi \cite{Renyi}. Every number $x
\in [0,1]$ can be given a $\beta$-expansion by the following {\em
greedy algorithm}:
\begin{equation}\label{Renyi_expansion}
r_0:=x {\rm ; \ for \ } j\ge 1 {\rm \ put \ } x_j:=\lfloor \beta r_{j-1}\rfloor {\rm \ and \ } r_j:= \beta r_{j-1} - x_j.
\end{equation}

\noindent Then $x=\sum_{j \ge 1} x_j \beta^{-j}$, and the digits
$x_j$ are  elements of the so-called {\em canonical alphabet}
$\C_\beta = \{ 0,1, \ldots, \lceil \beta \rceil -1\}$. For $x \in
[0,1)$, the sequence $(x_j)_{j \ge 1}$ is said to be the {\em
R\'enyi expansion} or the {\em $\beta$-greedy expansion} of $x$.

The greedy algorithm applied to the number~$1$ gives the
$\beta$-expansion of~$1$, denoted by $d_\beta(1)=(t_j)_{j\geq 1}$,
and plays a special role in this theory. We define also the {\em
quasi-greedy expansion} $d_\beta^*(1)=(t_j)_{j\geq 1}$ by: if
$d_\beta(1)=t_1 \cdots t_m$ is finite, then $d_\beta^*(1)=(t_1
\cdots t_{m-1}(t_m-1))^\omega$, otherwise
$d_\beta^*(1)=d_\beta(1)$. A number $\beta>1$ such that
$d_\beta(1)$ is eventually periodic, that is to say, of the form
$t_1 \cdots t_m(t_{m+1} \cdots t_{m+p})^\omega$ is called a {\em
Parry  number}. If $d_\beta(1)$ is finite, $d_\beta(1)=t_1 \cdots
t_m$, then $\beta$ is a {\em simple} Parry number.

Some numbers have more than one $(\beta,\C_\beta)
$-representation. The greedy  expansion of $x$ is
lexicographically the greatest among all $(\beta,\C_\beta)
$-representations of $x$.

A sequence $(x_j)_{j\geq 1}$ is said to be {\em
$\beta$-admissible} if it is the  greedy expansion of some $x \in
[0,1)$. Let us stress that not all sequences over the
alphabet~$\C_\beta $ are $\beta$-admissible.
Parry in ~\cite{Parry} used the {\em quasi-greedy expansion}
$d_\beta^*(1)=(t_j)_{j\geq 1}$ of~$1$ for characterization of
$\beta$-admissible sequences:  Let $s=(s_j)_{j\geq 1} =
s_1s_2s_3\cdots$ be an infinite sequence of non-negative integers.
The sequence $s$ is $\beta$-admissible if and only if for all $k
\ge 1$ the inequality $s_ks_{k+1} \cdots \prec_{lex}  d_\beta^*(1)$ holds in
the lexicographic order.

A $(\beta,\C_\beta) $-representation $x_nx_{n-1} \ldots x_0\bullet
x_{-1}x_{-2}\cdots$  of a number $x \geq 1$ is called
the $\beta$-greedy expansion of $x$, if the sequence $x_nx_{n-1}
\ldots x_0x_{-1}x_{-2}\cdots$ is $\beta$-admissible.

\bigskip

Some real bases introduced in~\cite{FrSo} have a property which is
interesting in  connection with parallel addition. A number
$\beta>1$ is said to {\em satisfy the (PF) Property} if the sum of
any two positive numbers with finite greedy $\beta$-expansions in
base~$\beta$ has a greedy $\beta$-expansion which is finite as
well, that is to say, every element of $\N[\beta^{-1}] \cap [0,1)$
has a finite greedy $\beta$-expansion. A number $\beta > 1$ is
said to {\em satisfy the (F) Property} if every element of
$\Z[\beta^{-1}] \cap [0,1)$ has a finite greedy $\beta$-expansion.
Of course, the (F) Property implies the (PF) Property.

If $\beta>1$ has the (PF) Property, then $\beta$ is a {\em Pisot
number}, i.e., $\beta$ is  an algebraic integer with  all its
algebraic conjugates of modulus strictly less than~$1$. But there
exist also Pisot numbers not satisfying the (PF) Property.

In \cite{FrSo}, two classes of Pisot numbers with the (PF) Property are presented:
\begin{itemize}
    \item $\beta$ has the (F) Property, and thus the (PF) Property as well, if $d_\beta(1) = t_1 t_2\cdots t_m$ and $t_1 \geq t_2 \geq \cdots \geq t_m \geq 1$.
    \item $\beta $ has the (PF) Property if $d_\beta(1) = t_1t_2\cdots t_mt^\omega$ and $t_1\geq t_2\geq \cdots \geq t_m > t \geq 1$.
\end{itemize}

In particular, every quadratic Pisot number satisfies the (PF) Property.

%%%%%%%%%%%%%%%%%%%%%%%%%%%%%%%%%%%%%%%%%%%%%%%%%%%%%%%%%%%%%%%%%%%%%%%%%%%%%%%%%%%%%%%%%

\subsection{Parallel addition}

%%%%%%%%%%%%%%%%%%%%%%%%%%%%%%%%%%%%%%%%%%%%%%%%%%%%%%%%%%%%%%%%%%%%%%%%%%%%%%%%%%%%%%%%%

Let us first formalize the notion of parallel addition as it is
considered in  most of works concentrated  on this topic,
including our recent papers.
\begin{definition}\label{local}
A function $\varphi : \A^{\Z} \rightarrow \B^{\Z}$  is said to be
{\em $p$-local}  if there exist two non-negative integers $r$ and
$t$ satisfying $p=r+t+1$, and a function $\Phi:\A^p\rightarrow \B$
such that, for any $u=(u_j)_{j \in \Z} \in \A^{\Z}$ and its image
$v=\varphi(u)=(v_j)_{j\in \Z} \in \B^{\Z}$, we have
$v_{j}=\Phi(u_{j+t} \cdots u_{j-r})$ for every $j$ in $\Z$.
\end{definition}

This means that the image of $u$ by $\varphi$ is obtained through
a sliding  window of length~$p$. The parameter $r$ is called the
\emph{memory} and the parameter $t$ is called the
\emph{anticipation} of the function $\varphi$. Such functions,
restricted to finite sequences, are computable by a parallel
algorithm in constant time.

\begin{definition}\label{digitsetconv}
Given a base~$\beta$ with $|\beta| >1$ and two alphabets $\A$ and
$\B$ of contiguous integers containing~$0$, a {\em digit set
conversion} in base~$\beta$ from $\A$ to $\B$ is a function
$\varphi: \A^\Z \rightarrow \B^\Z $ such that
\begin{enumerate}
    \item for any $u=(u_j)_{j \in \Z} \in \A^\Z$ with a finite number
    of non-zero digits, the image $v=(v_j)_{j\in \Z} = \varphi(u)\in \B^{\Z}$
    has only a finite number of non-zero digits as well, and
    \item $\sum\limits_{j\in \Z} v_j\beta^j = \sum\limits_{j\in \Z} u_j\beta^j$.
\end{enumerate}
Such a conversion is said to be {\em computable in parallel} if it is
a $p$-local function for some $p \in \N$.
\end{definition}

Thus, addition in ${\rm Fin}_{\A}(\beta)$ is computable in parallel if
there exists a digit set conversion in base~$\beta$ from $\A+\A$ to $\A$
which is computable in parallel.

Let us stress that all alphabets we use are formed by contiguous integers
 and contain~$0$. This restriction already forces the base~$\beta$ to be
 an algebraic number. In \cite{FrPeSv1} we give a sufficient condition
 on $\beta$ to allow parallel addition:

\begin{thm}\label{alphabetExists}
Let $\beta$ be an algebraic number such that $|\beta |>1$ and all
its conjugates  in modulus differ from~$1$. Then there exists an
alphabet~$\A$ of contiguous integers containing~$0$ such that
addition on ${\rm Fin}_{\A}(\beta)$ can be performed in parallel.
\end{thm}

The proof of the previous theorem  gives a method for finding a
suitable  alphabet~$\A$ and provides an algorithm for parallel
addition on this alphabet. But, in general, the alphabet~$\A$
obtained in this way is quite large. An exaggerated size of the
alphabet does not allow to compare numbers by means of the
lexicographic order on their $(\beta, \A)$-representations. For
instance, in base $\beta = 2$ and alphabet $\A = \{0,1,2\}$, we
have $02\prec_{lex}10$ in the lexicographic order, but $x= \bullet 02 \not<
y= \bullet 10$.

Therefore, in \cite{FrPeSv2}, we have studied the cardinality of
minimal  alphabets allowing parallel addition for a given
base~$\beta$. In particular, we have found the following lower
bounds:

\begin{thm}\label{zdola}
Let $\beta$, with $|\beta| > 1$, be an algebraic integer of
degree~$d$ with  minimal polynomial $f(X) = X^d - a_{d-1} X^{d-1}
- a_{d-2} X^{d-2} - \cdots - a_1 X - a_0$. Let $\A$ be an alphabet
of contiguous integers containing~$0$ and~$1$. If addition in
${\rm Fin}_{\A}(\beta)$ is computable in parallel, then $\# \A
\geq |f(1)|$. If, moreover, $\beta$ is a positive real number,
$\beta > 1$, then $\# \A \geq |f(1)| +2$.
\end{thm}

In \cite{Kornerup}, Kornerup suggested a more general concept of
parallel  addition. Instead of manipulating single digits, one
works with blocks of digits with fixed block length~$k$. For the
precise description of the Kornerup's idea, we introduce the
notation
\begin{equation}
\A_{(k)} = \{ a_0+a_1\beta +\cdots + a_{k-1}\beta^{k-1}\mid a_i \in \A\} \, ,
\end{equation}
where $\A$ is an alphabet and~$k$ a positive integer. Clearly, $\A_{(1)} = \A$.

\begin{definition}\label{kblock}
Given a base~$\beta$ with $|\beta| >1$ and two alphabets $\A$ and
$\B$ of  contiguous integers containing~$0$, a {\em digit set
conversion} in base~$\beta$ from $\A$ to $\B$ is said to be {\em
block parallel computable} if there exists some $k \in \N$ such
that the digit set conversion in base $\beta^k$ from $\A_{(k)}$ to
$ \B_{(k)}$ is computable in parallel. When the specification
of~$k$ is needed, we say {\em $k$-block parallel computable}.
\end{definition}

In this terminology,  the original parallel addition is $1$-block parallel addition.

\begin{remark}\label{int-base}
Suppose that the base is an integer $\beta$ with $|\beta| \ge 2$.
It is known  that $1$-block parallel addition is possible on an
alphabet of cardinality $\#\A = \beta+1$ (see \cite{Parhami} and
\cite{FrPeSv2}). But $k$-block parallel addition on an
alphabet~$\A$ is just $1$-block parallel addition in integer base
$\beta^k$ on $\A_{(k)}$. Thus $k$-block parallel addition in
integer base~$\beta$ can only be possible on an alphabet~$\A$ such
that $\#\A_{(k)} \ge \beta^k+1$. This shows that $k$-block
parallel addition with $k \ge 2$ does not allow the use of any
smaller alphabet than already achieved with $k=1$.
\end{remark}

The bound from Theorem~\ref{zdola} on the minimal cardinality of
alphabet~$\A$  cannot be applied to block parallel addition. This
fact can be demonstrated on the Penney numeration system with the
complex base $\beta = \imath -1$. The minimal polynomial of this
base is $X^2+2X+2$. From Theorem~\ref{zdola} we get that $1$-block
parallel addition in base $\imath-1$ requires an alphabet of
cardinality at least $5$, whereas Herreros in \cite{Herreros} gave
an  algorithm for $4$-block parallel addition on the alphabet
$\{-1,0,1\}$. According to our up-to-now knowledge, the base
$\beta =\imath-1$ is the only known example where the Kornerup
block approach  to  sequences of digits reduces the size of the
needed alphabet.

%%%%%%%%%%%%%%%%%%%%%%%%%%%%%%%%%%%%%%%%%%%%%%%%%%%%%%%%%%%%%%%%%%%%%%%%%%%%%%%%%%%%%%%%%
%%%%%%%%%%%%%%%%%%%%%%%%%%%%%%%%%%%%%%%%%%%%%%%%%%%%%%%%%%%%%%%%%%%%%%%%%%%%%%%%%%%%%%%%%

\section{Necessary conditions for  existence of block parallel addition}\label{necessary}

%%%%%%%%%%%%%%%%%%%%%%%%%%%%%%%%%%%%%%%%%%%%%%%%%%%%%%%%%%%%%%%%%%%%%%%%%%%%%%%%%%%%%%%%%

\subsection{General result}

%%%%%%%%%%%%%%%%%%%%%%%%%%%%%%%%%%%%%%%%%%%%%%%%%%%%%%%%%%%%%%%%%%%%%%%%%%%%%%%%%%%%%%%%%

In \cite{FrPeSv1} we have shown that the assumption that all the
algebraic  conjugates of $\beta$ have modulus different from~$1$
enables $1$-block parallel addition on ${\rm Fin}_{\A}(\beta)$ for
some suitable alphabet $\A \subset \Z$. The following theorem
shows that this assumption is also necessary and, even more, the
generalization of parallelism via working with $k$-blocks does not
change the situation.

%%%%%%%%%%%%%%%%%%%%%%%%%%%%%%%%%%%%%%%%%%%%%%%%%%%%%%%%%%%
\begin{thm}\label{conjugate1}
Let the base $\beta\in \mathbb{C}$, $|\beta| > 1$, be an algebraic number
with a  conjugate $\gamma$ of modulus $|\gamma|=1$ and let $\A
\subset \Z$ be an alphabet of contiguous integers containing~$0$.
Then addition on $\A$ cannot be block parallel computable.
\end{thm}

\begin{proof}
Within the proof, we denote by $\Re (x)$ the real part of a
complex number $x$.  Let us assume that there exist $k, p \in \N$
such that $\Phi: \A^p_{(k)} \to \A_{(k)}$ performs $k$-block
parallel addition on $\A$. Denote $S:= \max \Bigl\{
\Bigl|\sum\limits_{j=0}^{pk-1}a_j\gamma^j\Bigr| : a_j \in
\A\Bigr\}$. Since there exist infinitely many $j \in \N$ such that
$\Re (\gamma^j) > \frac12$, one can find $N > p$ and
$\varepsilon_j \in \{0,1\}$ such that $\Re
\Bigl(\sum\limits_{j=0}^{kN-1} \varepsilon_j \gamma^j\Bigr) > 3S$.

Let $T:=\max \{\bigl| \Re \bigl(\sum_{j=0}^{kN-1} b_j\gamma^j\bigr)\bigr|\ :\
b_j\in \A\}$.  Then find $x = \sum\limits_{j=0}^{kN-1} x_j\beta^j$
such that $\left|\Re (x')\right| = T$, where $x'$ denotes the image of $x$
under the field isomorphism $\mathbb{Q}(\beta) \to
\mathbb{Q}(\gamma)$. The choice of $N$ ensures $\left|\Re (x')\right| > 3S$. Adding $x+x$ by the
$k$-block $p$-local function $\Phi$, we get
$$ x+x = \sum_{j=kN}^{k(N+p)-1} z_j\beta^j + \sum_{j=0}^{kN-1} z_j\beta^j+ \sum_{j=-kp}^{-1} z_j\beta^j\,, \ \ \hbox{with } z_j \in \A\,.$$
For the image of $x+x$ under the field isomorphism, we have
$$\left|\Re  (x')\right| + 3S < \left|\Re  (x'+x')\right| \leq |\gamma^{kN}|S +  \left|\Re (x')\right| + |\gamma^{-kp}|S \leq 2S + \left|\Re (x')\right|, $$
which is a contradiction.
\end{proof}
%%%%%%%%%%%%%%%%%%%%%%%%%%%%%%%%%%%%%%%%%%%%%%%%%%%%%%%%%%%

%%%%%%%%%%%%%%%%%%%%%%%%%%%%%%%%%%%%%%%%%%%%%%%%%%%%%%%%%%%%%%%%%%%%%%%%%%%%%%%%%%%%%%%%%

\subsection{Positive real bases}

%%%%%%%%%%%%%%%%%%%%%%%%%%%%%%%%%%%%%%%%%%%%%%%%%%%%%%%%%%%%%%%%%%%%%%%%%%%%%%%%%%%%%%%%%

Since the integer base case has been resolved in
Remark~\ref{int-base}, in the  following we suppose that $\beta$
is not an integer.

For positive bases $\beta$ belonging to some classes of Parry
numbers  we deduce  lower bound on the size of the alphabet $\A
\subset \N$ allowing block parallel addition. For a non-negative
alphabet we utilize the well known properties of the greedy
representations, which are in the lexicographic order the greatest
ones among all representations. At first we state a simple
observation we will use in our later considerations.

%%%%%%%%%%%%%%%%%%%%%%%%%%%%%%%%%%%%%%%%%%%%%%%%%%%%%%%%%%%
\begin{lemma}\label{coincidence}
Let $\beta >1$ be a base and let $\A = \{0,1,\ldots, M\}$ with $M
\ge 1$ be an  alphabet. Let $z= g_0\bullet g_1g_2 \cdots $ be a
$(\beta, \A)$-representation of $z$ such that there exists $n \ge
0$ such that for $0 \le i \le n$ the inequality
\begin{equation}\label{MtoOmega}
1\bullet g_{i+1}g_{i+2}g_{i+3} \cdots > 0\bullet M^\omega
\end{equation}
holds true. Then any lexicographically smaller $(\beta,
\A)$-representation of $z$  coincides with the original
representation on the first $n+1$ digits, i.e., it has the form $z
= g_0\bullet g_1g_2 \cdots g_nz_{n+1}z_{n+2}\cdots$.
\end{lemma}

\begin{proof}
Let $z = z_0\bullet z_1z_2 \cdots z_nz_{n+1}z_{n+2}\cdots$ be a
lexicographically  smaller representation of $z$ and $i$ be the
minimal index for which $z_i < g_{i}$. Then
$$ 0\bullet M^\omega \geq 0 \bullet z_{i+1}z_{i+2} \cdots = (g_i-z_i)\bullet g_{i+1}g_{i+2} \cdots  \geq 1\bullet g_{i+1}g_{i+2} \cdots. $$
Since for $i\leq n$ the opposite inequality \eqref{MtoOmega} holds, necessarily $i\geq n+1$.
\end{proof}
%%%%%%%%%%%%%%%%%%%%%%%%%%%%%%%%%%%%%%%%%%%%%%%%%%%%%%%%%%%

%%%%%%%%%%%%%%%%%%%%%%%%%%%%%%%%%%%%%%%%%%%%%%%%%%%%%%%%%%%%%%%%%%%%%%%%%%%%%%%%%%%%%%%%%

\subsubsection{Simple Parry numbers}\label{Parry_simple}

%%%%%%%%%%%%%%%%%%%%%%%%%%%%%%%%%%%%%%%%%%%%%%%%%%%%%%%%%%%%%%%%%%%%%%%%%%%%%%%%%%%%%%%%%

%%%%%%%%%%%%%%%%%%%%%%%%%%%%%%%%%%%%%%%%%%%%%%%%%%%%%%%%%%%
\begin{thm}\label{Simple}
Let $d_\beta(1) = t_1t_2\cdots t_m$ with $m \ge 2$ and $1\leq t_m
\leq t_i$ for  $1 \le i \le m$ be the R\'enyi expansion of~$1$ in
non-integer base~$\beta$. If block parallel addition can be
performed on alphabet $\A = \{0,1,\ldots, M\}$, then $M\geq
t_1+t_m$.
\end{thm}

\begin{proof}
Let $\Phi:{\A^p_{(k)}} \to \A_{(k)}$ be the function performing $k$-block
 parallel addition on the alphabet $\A = \{0,1,\ldots, M\}$. Let us suppose
 that the theorem does not hold. Without loss of generality, we can  set $M= t_1+t_m-1$.\\

Since $d_\beta(1) = t_1t_2\cdots t_m$, one can easily find other representations
of~$1$, namely
$$1\bullet = \bullet \bigl(t_1t_2\cdots t_{m-1}(t_m-1)\bigr)^\omega$$
and
\begin{equation}\label{unit}
\hbox{for\ } n\in\N, \quad 1\bullet = \bullet \bigl(t_1t_2\cdots t_{m-1}(t_m-1)\bigr)^n t_1t_2\cdots t_m.
\end{equation}
Denote the periodic factor by
\begin{equation}\label{Per}
{\rm Per}=t_1t_2\cdots t_{m-1}(t_m-1),
\end{equation}
as it will be used in the sequel several times. The value $\bullet M^\omega$ is the
largest fractional part one can obtain in our alphabet, as the base is
positive. The other representation on $\A$ of $\bullet M^\omega$ is
$$ \bullet M^\omega = 1\bullet \bigl((M-t_1)(M-t_2)\cdots (M-t_{m-1}) (M-t_{m}+1)\bigr)^\omega\,.$$
As $t_1\geq t_i\geq  t_m$, we have $0\leq M-t_i = t_1 +t_m-1 - t_i \leq t_1-1$ for $i=2,3,\ldots, m-1$,
further $M-t_1 = t_m-1 < t_2$ and $M-t_{m}+1 = t_1$. Thus the representation on the right side is
the greedy one. Moreover, because of $(M-t_1)(M-t_2) = (t_m-1)(t_1+t_m-t_2-1) \prec_{lex} (t_m-1)t_1 \prec_{lex}  t_m$, we have in particular
\begin{equation}\label{boundMomega}
0\bullet M^\omega < 1\bullet (t_m-1)t_1 < 1\bullet t_m
\end{equation}

%%%%%%%%%%%%%%%%%%%%%%%%%%%%%%%%%%%%%%%%%%%%
\noindent\textbf{Statement 0}:\quad {\it The only finite
representations of~$1$ in  base~$\beta$ on $\A$ are listed in
\eqref{unit}.}

\begin{proof}
Let us denote the digits of the string $({\rm Per})^\omega$ by
$t_1^*t_2^*t_3^*\cdots$,  i.e., $t_i^* = t_i-1$ if $i = 0 \mod m$,
and $t_i^* = t_i$ otherwise. Clearly $1\bullet = \bullet
t_1^*t_2^*t_3^*\cdots$ is an infinite representation of~$1$. Using
\eqref{boundMomega}, we have
$$1\bullet t_i^*t_{i+1}^*t_{i+2}^*\cdots > 1\bullet t_i^*  = 1\bullet t_i \geq 1\bullet t_m > 0\bullet M^\omega \quad \hbox{if } \ i \neq 0 \mod m$$
and, analogously,
$$1\bullet t_i^*t_{i+1}^*t_{i+2}^*\cdots > 1\bullet (t_m-1)t_1  > 0\bullet M^\omega  \quad \hbox{if } \ i = 0 \mod m$$
Applying Lemma \ref{coincidence}, we can conclude that any other
representation $0\bullet x_1x_2x_3 \cdots $ of~$1$ must be
lexicographically bigger than $0\bullet t_1^*t_2^*t_3^*\cdots$.
Let us denote by~$k$ the smallest index such that $x_k > t_k^*$.
Obviously, $x_{k}\bullet x_{k+1}x_{k+2}x_{k+3}\cdots =
t_k^*\bullet t_{k+1}^*t_{k+2}^*t_{k+3}^*\cdots$, and therefore
$$1\bullet \leq (x_k-t_k^*)\bullet x_{k+1}x_{k+2}x_{k+3}\cdots = 0\bullet t_{k+1}^*t_{k+2}^*t_{k+3}^*\cdots \leq 0\bullet t_1^*t_2^*t_3^*\cdots = 1\bullet , $$
 where the last inequality follows from the Parry condition. As both tails of the previous
 row are equal to the same number, the inequalities can be replaced by
equalities. In particular, it means that $0\bullet x_{k+1}x_{k+2}x_{k+3}\cdots = 0$
and $x_k= t^*_k+1$ and $({\rm Per})^\omega =  t_1^*t_2^*t_3^*\cdots  = (t_1^*t_2^*\cdots t_{k}^*)^\omega$.
And thus~$k$ is a multiple of $m$, as desired.
\end{proof}
%%%%%%%%%%%%%%%%%%%%%%%%%%%%%%%%%%%%%%%%%%%%

Fix $n \in \N$. During the course of the proof we will work with the
 following two numbers:
\begin{equation}\label{z_y}
z=\bullet ({\rm Per})^n t_1 \quad \hbox{and} \quad y=\bullet (M+1)t_2t_3\cdots t_{m-1}(t_m-1)({\rm Per})^n t_1t_2\cdots t_m.
\end{equation}

First, we show three auxiliary statements about numbers $z$ and $y$.\\[2mm]

%%%%%%%%%%%%%%%%%%%%%%%%%%%%%%%%%%%%%%%%%%%%
\noindent\textbf{Statement 1}:\quad {\it Any representation of $z= \bullet({\rm Per})^n t_1$
in base~$\beta$ on alphabet $\A = \{0,1,\ldots, M\}$ has the form $\bullet ({\rm Per})^n z_{mn+1} z_{mn+2}\cdots$\,.}
\begin{proof}
Since $\bullet({\rm Per})^n t_1$ is the greedy representation of $z$, any other
representation is lexicographically smaller. According to Lemma \ref{coincidence},
it is enough to show that for any~$k$, $0 \leq k < n$,
\begin{equation}\label{middle}
1\bullet t_j\cdots t_{m-1}(t_m-1) ({\rm Per})^k t_1 > 0\bullet M^\omega \quad \hbox{for any } j=1,2, \ldots, m-1,
\end{equation}
and
\begin{equation}\label{ending}
1\bullet (t_m-1) ({\rm Per})^k t_1 > 0\bullet M^\omega .
\end{equation}
Both inequalities follow from \eqref{boundMomega} and the assumption $t_m\leq t_i$.
\end{proof}
%%%%%%%%%%%%%%%%%%%%%%%%%%%%%%%%%%%%%%%%%%%%

%%%%%%%%%%%%%%%%%%%%%%%%%%%%%%%%%%%%%%%%%%%%
\noindent\textbf{Statement 2}:\quad {\it The greedy expansion of $y=\bullet (M+1)t_2t_3\cdots t_{m-1}(t_m-1)({\rm Per})^n t_1t_2\cdots t_m$ in base~$\beta$ is $1\bullet t_m$\,.}
\begin{proof}
The statement follows from the fact $1\bullet  = \bullet ({\rm Per})^{n+1}t_1t_2\cdots t_m$.
\end{proof}
%%%%%%%%%%%%%%%%%%%%%%%%%%%%%%%%%%%%%%%%%%%%

%%%%%%%%%%%%%%%%%%%%%%%%%%%%%%%%%%%%%%%%%%%%
\noindent\textbf{Statement 3}:\quad {\it Any finite non-greedy representation of
$y$ in base~$\beta$ on alphabet $\A = \{0,1,\ldots, M\}$
has the form $1\bullet (t_m-1)({\rm Per})^\ell t_1t_2\cdots t_m$ with $\ell \geq 0$ in $\N$.}

\begin{proof}
It follows from  Statement~0.
\end{proof}
%%%%%%%%%%%%%%%%%%%%%%%%%%%%%%%%%%%%%%%%%%%%

Let us now finish the proof of the theorem. For all $n \in \N$, according to Statement~1,
 the sequence $0\bullet ({\rm Per})^{n}t_1$ has to be rewritten by the local function $\Phi$
 into the sequence $0\bullet ({\rm Per})^{n}w$, where $w \in \A^*$. It means that the
 periodic word ${\rm Per}$ starts at the same
positions (namely $1+mi$ for $i = 0,1, \ldots, n-1$) after the $\bullet$ in the original
string as well as in the string rewritten by the function $\Phi$.

Consider now the sequence $\bullet (M+1)t_2t_3\cdots t_{m-1}(t_m-1)({\rm Per})^n t_1t_2\cdots t_m$.
Let us stress that the length of the preperiod $(M+1)t_2t_3\cdots t_{m-1}(t_m-1)$ is the same
 as the length of the period ${\rm Per}$, and thus the string ${\rm Per}$ starts at the positions $1+mi$ for $i = 0,1, \ldots, n$.

According to Statement~3, the sequence $\bullet (M+1)t_2t_3\cdots t_{m-1}(t_m-1)({\rm Per})^n t_1t_2\cdots t_m$
has to be rewritten into $1\bullet t_m 0^\omega$ or into $1\bullet (t_m-1) ({\rm Per})^{\ell} t_1\cdots t_m$
for some $\ell \in \N$, i.e., the string  ${\rm Per}$ starts at the positions $2+mi$. Since
${\rm Per}$ is not a power of a single letter, no such local function $\Phi$ can exist.
\end{proof}
%%%%%%%%%%%%%%%%%%%%%%%%%%%%%%%%%%%%%%%%%%%%%%%%%%%%%%%%%%%

\bigskip

We will illustrate that the lower bound on the cardinality of the alphabet in Theorem~\ref{Simple}
is sharp, i.e. can be attained, in quadratic cases. In order to do so, we exploit the positive root
 of the equation $X^2=aX+b$. We first assume that $a \ge b+2$ and $b \ge 2$.

%%%%%%%%%%%%%%%%%%%%%%%%%%%%%%%%%%%%%%%%%%%%%%%%%%%%%%%%%%%
\begin{proposition}\label{SimpleBound}
Let $d_\beta(1) = ab$, where $a \ge b+2$ and $b \ge 2$, be the R\'enyi expansion of~$1$ in base~$\beta$.
Then $1$-block parallel addition in base~$\beta$ is possible on alphabet $\A=\{0, \ldots, a+b\}$.
\end{proposition}

By Proposition~18 in~\cite{FrPeSv2}, it is enough to show that the greatest digit elimination
from $\{0,\ldots, a+b+1\}$ to $\{0, \ldots, a+b\} = \A$ can be done in parallel:
%%%%%%%%%%%%%%%%%%%%%%%%%%%%%%%%%%%%%%%%%%%%%%%%%%%%%%%%%%%
\vskip0.2cm \hrule \vskip0.2cm

\noindent {\bf Algorithm~\textsl{GDE}($\beta^2=a\beta+b$)}: Base $\beta > 1$ satisfying
$\beta^2=a\beta+b$, $a \ge b+2$, $b \ge 2$, parallel conversion (greatest digit elimination)
 from $\{ 0,\ldots, a+b+1\}$ to $\{0, \ldots, a+b\} = \A$.

\vskip0.2cm \hrule \vskip0.2cm

\noindent{\sl Input}: a finite sequence of digits $(z_j)$ of  $\{0,\ldots, a+b+1 \}$, with $ z= \sum z_j\beta^j$.\\
{\sl Output}: a finite sequence of digits $(x_j)$ of $\{0,\ldots,
 a+b\}$, with $z = \sum x_j\beta^j$.

\vskip0.2cm

\noindent\texttt{for each $j$ in parallel do}\\

1. \hspace*{0.3cm} \texttt{case}
    $\left\{\begin{array}{l}
        z_j = a+b+1\  \\
        z_j = a+b  \hbox{ \texttt{and}} \ \bigl( z_{j+1} \leq  b-1 \ \hbox{ \texttt{or}}\ z_{j-1}\geq a \Bigr) \ \\
        a+1 \le z_j \le a+b-1\  \hbox{ \texttt{and}} \  z_{j+1} \leq  b-1  \ \\
        z_j = a  \hbox{ \texttt{and}} \  z_{j+1} \leq  b-1 \ \hbox{ \texttt{and}}\ z_{j-1}\geq a  \ \\
    \end{array} \right\} $
\texttt{then} $q_j:=1$\\[1mm]
\hspace*{1.5cm} \texttt{if} $z_j  \le b-1 \hbox{ \texttt{and}} \  z_{j+1} \geq  a$ \ \texttt{then} $q_j:=-1$\\[1mm]
\hspace*{1.5cm} \texttt{else} $q_j:=0$\\

2. \hspace*{0.3cm} $x_j := z_j - a q_j - b q_{j+1} + q_{j-1}$

\vskip0.2cm \hrule \vskip0.2cm
%%%%%%%%%%%%%%%%%%%%%%%%%%%%%%%%%%%%%%%%%%%%%%%%%%%%%%%%%%%

\begin{proof}
Let us denote $w_j := z_j- a q_j$, and inspect all the possible combinations of $(z_{j+1}, z_j, z_{j-1})$ that can occur:
\begin{itemize}
    \item $z_j=a+b+1$: Then $w_j=b+1$, and $q_{j+1}, q_{j-1} \in \{-1,0,1\}$, thus $0 \le x_j \le 2b+2 \le a+b$, since $a \ge b+2$.

    \item $z_j=a+b$ and $z_{j+1} \leq b-1$: Then $w_j=b$, $q_{j+1} \in \{-1,0\}$, and $q_{j-1} \in \{-1,0,1\}$, thus $0 \le b-1 \le x_j \le 2b+1 < a+b$, since $a \ge b+2$.

    \item $z_j=a+b$ and $z_{j-1} \geq a$: Then $w_j=b$, $q_{j-1} \in \{0,1\}$, and $q_{j+1} \in \{-1,0,1\}$, thus $0 \le x_j \le  2b+1 \le a+b$, since $a \ge b+2$.

    \item $z_j=a+b$ and $z_{j+1} \ge b$ and $z_{j-1} \le a-1$: Then $w_j=a+b$, $q_{j+1} \in \{0,1\}$, and $q_{j-1} \in \{-1,0\}$, so $0 < a-1 \le x_j \le  a+b$.

    \item $a+1 \le z_j \le a+b-1$ and $z_{j+1} \leq b-1$: Then $1\le w_j \le b-1$, $q_{j+1} \in \{-1,0\}$, and $q_{j-1} \in \{-1,0,1\}$, thus $0\le x_j \le  2b < a+b$.

    \item $a+1 \le z_j \le a+b-1$ and $z_{j+1} \ge b$: Then $a+1 \le w_j=z_j \le a+b-1$, $ q_{j+1} \in \{0,1\}$, and $q_{j-1} \in \{-1,0,1\}$, so $0 < a-b \le x_j \le a+b$.

    \item $z_j=a$ and $z_{j+1} \leq b-1$ and $z_{j-1}\geq a$: Then $w_j=0$, $q_{j+1} \in \{-1,0\}$, and $q_{j-1} \in \{0,1\}$, thus $0 \le x_j \le b+1 < a+b$.

    \item $z_j=a$ and $z_{j+1} \ge b$: Then $w_j=z_j=a$, $q_{j+1} \in \{0,1\}$, and $q_{j-1} \in \{-1,0,1\}$, thus $0 < a-b-1 \le x_j \le a+1 \le a+b$.

    \item $z_j=a$ and $z_{j-1} \le a-1$: Then $w_j=z_j=a$, $q_{j-1} \in \{-1,0\}$, and $q_{j+1} \in \{-1,0,1\}$, thus $0 < a-b-1 \le x_j \le a+b$.

    \item $b \le z_j \le a-1$: Then $b \le w_j=z_j \le a-1$, $q_{j+1} \in \{-1,0,1\}$, and $q_{j-1} \in \{0,1\}$, thus $0 \le x_j \le a+b$.

    \item $z_j \le b-1$ and $z_{j+1} \geq a$: Then $a \le w_j \le a+b-1$, and $q_{j+1}, q_{j-1} \in \{0,1\}$, thus $0 < a-b \le x_j \le a+b$.

    \item $z_j \le b-1$ and $z_{j+1} \leq a-1$: Then $0 \le w_j=z_j \le b-1$, $q_{j+1} \in \{-1,0\}$, and $q_{j-1} \in \{0,1\}$, so $0 \le x_j \le 2b < a+b$.
\end{itemize}
It is also obvious that a string of zeroes cannot be converted by the local function in this
algorithm into a string of non-zeroes, therefore, the algorithm performs a correct digit set conversion.
\end{proof}
%%%%%%%%%%%%%%%%%%%%%%%%%%%%%%%%%%%%%%%%%%%%%%%%%%%%%%%%%%%

The previous algorithm acts on alphabet $\mathcal{{A}}\subset \N$. Looking for the letters
 $h \in \A =\{0, \ldots, a+b\}$ such that the algorithm keeps unchanged the constant
 sequences $(h)_{j \in \Z}$ allows us to modify the alphabet of the algorithm:

\begin{definition}\label{fixed}
Let $\A$ and $\B$ be two alphabets containing~$0$ such that $\A \cup \B \subset \Z[\beta]$.
 Let $\varphi: \A^\Z \rightarrow \B^\Z$ be a $p$-local function
realized by the function $\Phi: \A^p \rightarrow \B$. The \emph{letter $h$} in $\A$ is
said to be \emph{fixed by $\varphi$} if $\varphi((h)_{j \in \Z}) = (h)_{j \in \Z} $, or, equivalently, $\Phi(h^p)=h$.
\end{definition}

%%%%%%%%%%%%%%%%%%%%%%%%%%%%%%%%%%%%%%%%%%%%%%%%%%%%%%%%%%%
\begin{proposition}
Let $\beta > 1$ satisfy $\beta^2 = a\beta +b$, with $a \ge b+2$, $b \ge 2$. Parallel addition in
base~$\beta$ is possible on any alphabet of cardinality $a+b+1$ of contiguous integers containing~$0$.
\end{proposition}

\begin{proof}
Every letter $h$, $0 \le h \le a+b-1$, is fixed by the Algorithm~\textsl{GDE}($\beta^2=a\beta+b$)
above. So, for any $d = 1, \ldots, a+b-1$, both letters $d$ and $a+b-d$ are fixed by the algorithm,
and, by Corollary~24 in~\cite{FrPeSv2}, parallel addition is possible on any alphabet of
the form $\{-d, \ldots, a+b-d\} = \A$, with $d \in \{0, \ldots, a+b\}$.
\end{proof}
%%%%%%%%%%%%%%%%%%%%%%%%%%%%%%%%%%%%%%%%%%%%%%%%%%%%%%%%%%%

\begin{remark}
The  algorithm used in the proof of Proposition~\ref{SimpleBound} requires the coefficient $b$ in
the quadratic polynomial $X^2-aX-b$ to satisfy $a \ge b+2$, $b \ge 2$, but we have results also for other cases:
\begin{itemize}
    \item The case $b=1$ for $a \ge b$ is studied in \cite{FrPeSv2}, where we gave an algorithm
    for $1$-block parallel addition on the alphabet $\A = \{0,1,\ldots, a+1\}$, i.e., the bound of Theorem~\ref{Simple} is attained here, too.

    \item Also in the case of $b \ge 2$ with $a=b+1$, the lower bound from Theorem~\ref{Simple}
    on the cardinality of $\A$ is attained; moreover, with $k=1$. We can perform $1$-block parallel addition by the refined Algorithm~\textsl{GDE}($\beta^2=a\beta+a-1$) described below.

    \item For $b \ge 2$ and $a=b$, the lower bound on the cardinality of the alphabet~$\A$
     from Theorem~\ref{Simple} is attained as well. It follows from  Corollary~\ref{ConclusionSimple},
      where the existence of $k$-block parallel addition for this case is guaranteed. Besides, it is
      assumed that also here the parallel addition on alphabet of the minimal cardinality $\# \A = 2a+1$
      should be possible with $k=1$, i.e. $1$-block parallel, but the algorithm is a lot more
      complicated than for the case of $a=b+1$, and it still remains as an open task.
\end{itemize}
\end{remark}

%%%%%%%%%%%%%%%%%%%%%%%%%%%%%%%%%%%%%%%%%%%%%%%%%%%%%%%%%%%
\vskip0.2cm \hrule \vskip0.2cm

\noindent {\bf Algorithm~\textsl{GDE}($\beta^2=a\beta+a-1$)}: Base $\beta > 1$
satisfying $\beta^2=a\beta+a-1$, $a \ge 3$, parallel conversion (greatest digit elimination) from $\{ 0,\ldots, 2a\}$ to $\{0, \ldots, 2a-1\} = \A$.

\vskip0.2cm \hrule \vskip0.2cm

\noindent{\sl Input}: a finite sequence of digits $(z_j)$ from $\{0, \ldots, 2a\}$, with $z = \sum_j z_j\beta^j$.\\
{\sl Output}: a finite sequence of digits $(x_j)$ from $\{0, \ldots, 2a-1\}$, with $ z = \sum_j x_j \beta^{j}$.

\vskip0.2cm

\noindent\texttt{for each $j$ in parallel do}\\

1. \hspace*{0.3cm} \texttt{case}
    $\left\{\begin{array}{l}
        z_j = 2a\ \hbox{ \texttt{and}} \ z_{j+1} \le 2a-1 \\
        z_j = 2a\ \hbox{ \texttt{and}} \ z_{j+1} = 2a \hbox{ \texttt{and}} \ a \le z_{j+2} \\
        z_j = 2a-1\ \hbox{ \texttt{and}} \ z_{j+1}\le a-1 \\
        z_j = 2a-1\ \hbox{ \texttt{and}} \ a \le z_{j+1} \le 2a-1 \hbox{ \texttt{and}} \ a \le z_{j-1}  \\
        z_j = 2a-1\ \hbox{ \texttt{and}} \ z_{j+1} = 2a \hbox{ \texttt{and}} \ a \le z_{j+2}  \hbox{ \texttt{and}} \ a \le z_{j-1}  \\
        a+1 \le z_j \le 2a-2\ \hbox{ \texttt{and}} \ z_{j+1} \le a-1 \\
        z_j = a\ \hbox{ \texttt{and}} \ z_{j+1}  \le a-1  \hbox{ \texttt{and}} \ a \le z_{j-1}  \\
    \end{array} \right\} $  \texttt{then} $q_j := 1$\\[1mm]
\hspace*{1.5cm} \texttt{if} $z_j \le  a-2 \ \hbox{ \texttt{and}} \ a \le z_{j+1} $ \texttt{then} $q_j := -1$\\[1mm]
\hspace*{1.5cm} \texttt{else} $q_j:=0$\\

2. \hspace*{0.3cm} $x_j := z_j - a q_j  - (a-1) q_{j+1} + q_{j-1}$

\vskip0.2cm \hrule \vskip0.2cm
%%%%%%%%%%%%%%%%%%%%%%%%%%%%%%%%%%%%%%%%%%%%%%%%%%%%%%%%%%%

We present the Algorithm~\textsl{GDE}($\beta^2=a\beta+a-1$) without proving its correctness
in detail, as it is rather tedious. It can be proved by inspecting all the possible combinations
 of digits $(z_{j+1},z_j,z_{j-1})$, similarly as done above for the Algorithm~\textsl{GDE}($\beta^2=a\beta+b$).

\bigskip

Let us now consider a class of well studied Pisot numbers, generalizing the (quadratic) golden mean:

\begin{definition}  Let  $d \in \N, d \ge 2$. The real root $\beta >1$  of the
 equation $X^d = X^{d-1} + X^{d-2} + \cdots + X + 1$ is said to be the {\em $d$-bonacci number}.
  Specifically, the $2$-bonacci number (the golden mean) is called the {\em Fibonacci number}, and the $3$-bonacci number is called the {\em Tribonacci number}.
\end{definition}

As a corollary of Theorem~\ref{Simple}, we get the following result:

\begin{corollary}\label{Dbonacci}
Let $\beta$ be the $d$-bonacci number, $d \geq 2$. There exists no $k$-block $p$-local
function performing parallel addition in base~$\beta$ on the alphabet $\A = \{0,1\}$.
\end{corollary}

%%%%%%%%%%%%%%%%%%%%%%%%%%%%%%%%%%%%%%%%%%%%%%%%%%%%%%%%%%%%%%%%%%%%%%%%%%%%%%%%%%%%%%%%%

\subsubsection{Non-simple Parry numbers}\label{Parry_non-simple}

%%%%%%%%%%%%%%%%%%%%%%%%%%%%%%%%%%%%%%%%%%%%%%%%%%%%%%%%%%%
\begin{thm}\label{nonSimple} Let $d_\beta(1) = t_1t_2\cdots t_m(t_{m+1}t_{m+2} \cdots t_{m+p})^\omega$
 be the R\'enyi expansion of~$1$ in base~$\beta$. Let the coefficients $t_1,\ldots, t_{m+p}$ satisfy one of the following assumptions:
\begin{enumerate}
    \item $m=p=1$;
    \item $m=1$, $p \ge 2$, and $t_1 > t_2 > t_j$ for all~$j$ such that $2 < j \leq p+1$;
    \item $m \geq 2$ and $t_1 > t_2 \geq t_j$ for all~$j$ such that $2 \leq j \leq m$ and $t_2 > t_j$ for all~$j$ such that $m+1 \leq j \leq m+p$.
\end{enumerate}
If block parallel addition in base~$\beta$ can be performed on alphabet $\A = \{0, 1, \ldots, M\}$, then $M \geq 2t_1-t_2-1$.
\end{thm}

\begin{proof}
Note that, due to the fact that the R\'enyi expansion can never take the form $t_1^\omega$,
the assumptions imply that $t_1 > t_j$ for $j = 2, ..., m+p$ in all the three cases.

Let us first prove the inequality
\begin{equation}\label{ProofNonsimple1}
1 \bullet (t_1-1) > 0 \bullet (2t_1-t_2-2)^\omega.
\end{equation}
If $a$ is a digit, we use the notation $\overline{a} = -a$. The R\'enyi expansion of unity gives
the representation of number~$1$ in the form $1\bullet = \bullet t_1 t_2 t_3 \cdots $; consequently,
 we have a representation of zero in the form $0 = \overline{1} \bullet t_1 t_2 t_3 \cdots $.
 We will add to the left side of \eqref{ProofNonsimple1} the value $\overline{1}\bullet t_1$
 (which is negative, since $\overline{1}\bullet t_1 < \overline{1} \bullet t_1 t_2 t_3 \cdots = 0$)
 and infinitely many negative values $\bullet \overline{2} (2t_1)$, $\bullet 0\overline{2} (2t_1)$, $\bullet 00\overline{2} (2t_1)$, $\bullet 000\overline{2} (2t_1)$, etc. We obtain
$$ 1\bullet (t_1-1) > 0\bullet  (2t_1-3)(2t_1 -2)^\omega\,.$$
By assumption, we have $t_2 \geq 1$, and thus $0\bullet (2t_1-3)(2t_1-2)^\omega > 0\bullet (2t_1-t_2-2)^\omega$.
This proves the inequality \eqref{ProofNonsimple1}.

Now to prove the theorem by contradiction, let us put  $M= 2t_1-t_2-2$ and suppose that
conversion from $\{0, 1, \ldots M+1\}$ into $\{0, 1, \ldots, M\} = \A$ is possible in
parallel by a $k$-block $p$-local function $\Phi$. Let us denote the periodic part of
the R\'enyi expansion $d_\beta(1)$ by ${\rm Per} = t_{m+1}t_{m+2}\cdots t_{m+p}$. Find an
integer $\ell$ such that $p\ell > m$, and denote the digits of ${\rm Per}^\ell$ by
${\rm Per}^\ell=p_1p_2\cdots p_{p\ell}$. The string $P' = p'_1p'_2\cdots p'_{p\ell} $ is
then defined as just a small modification of ${\rm Per}^\ell$, namely as
$$ p'_j = \left\{ \begin{array}{ll}
            p_j+t_1-1 & \hbox{for\ } j = 1, \ldots, p\ell-m-1 \,,\\
            p_j+t_1 & \hbox{for\ } j = p\ell-m \,,\\
            p_j & \hbox{for\ } j = p\ell-m+1, \ldots, p\ell \,.
        \end{array}\right. $$

For a chosen integer $n\in\N$, we select two different strings $z = \bullet (M+1)(t_1-1)^{pn+m-1}{P'}$
and $y = \bullet (t_1-1)^{n+1}$, and convert them in various ways into $(\beta, \A)$-representations.

\smallskip

\noindent{\bf A)} \quad The string $z = \bullet (M+1)(t_1-1)^{pn+m-1}{P'}$ shall be converted as follows:

\smallskip

We easily find another representation of $z$, namely its $(\beta, \A)$-representation,
by adding suitable representations of~$0$ to the original string:
$$
\begin{array}{ccrccccc}
z & = & 0 \bullet & (M+1) &(t_1-1)^{m-1} & (t_1-1)^{pn} & P' & \\
0 & = & 1 \bullet & \overline{t}_1 & \overline{t}_2 \ldots \overline{t}_m & \overline{{\rm Per}}^n & \overline{{\rm Per}}^{\ell} & \overline{{\rm Per}}^{\,\omega} \\
0 & = & 0 \bullet & 0 & 0^{m-1} & 0^n & 0^{p\ell-m-1} \overline{1} t_1\ldots t_m & {\rm Per}^\omega \\
\hline
z & = & 1 \bullet & (M\!\!+\!\!1\!\!-\!\!t_1) & (t_1\!\!-\!\!1\!\!-\!\!t_2) \cdots (t_1\!\!-\!\!1\!\!-\!\!t_m) & H^n & (t_1\!\!-\!\!1)^{p\ell-m}\ t_1 \ldots t_m & 0^\omega
\end{array}\,,
$$
where $H=(t_1-1-t_{m+1}) \cdots (t_1-1-t_{m+p})$. In the last row of the table above,
we have expressed $z$ on the alphabet $\A = \{0, 1, \ldots M\}$. Let us denote this representation by
\begin{equation}\label{ProofNonsimple2}
z = g_0\bullet g_1 g_2 \cdots g_{p\ell+pn+m} \in {\rm Fin}_{\A}(\beta)\,.
\end{equation}
As $M=2t_1-t_2-2$, any $g_i < t_1$ and thus the representation in \eqref{ProofNonsimple2}
 satisfies the Parry condition, i.e., the representation $z = g_0 \bullet g_1 g_2 \cdots g_{p\ell+pn+m}$ is greedy.

Let us now show an auxiliary statement:

\smallskip

%%%%%%%%%%%%%%%%%%%%%%%%%%%%%%%%%%%%%%%%%%%%
\noindent\textbf{Statement 1}:\quad {\it Any representation of $z$ in base~$\beta$ on
alphabet~$\A$ other than \eqref{ProofNonsimple2} begins with the prefix $g_0\bullet g_1\cdots g_{pn+m}$.}

\begin{proof}
Since any other representation of $z$ must be lexicographically smaller than the greedy
representation $g_0\bullet g_1 g_2 \cdots g_{p\ell+pn+m}$, according to Lemma~\ref{coincidence} it is enough to show that
\begin{equation}\label{ProofNonsimple3}
0 \bullet M^\omega < 1\bullet g_{i+1}g_{i+2}g_{i+3}\cdots \quad \hbox{for}\ \ i=0,1,\ldots, pn+m\,.
\end{equation}
In particular, for $i=0$ we have to check that $\bullet M^\omega < 1\bullet g_1 \cdots g_{p\ell+pn+m} = z$.
 Since $z = \bullet (M+1)(t_1-1)^{pn+m-1}{P'}$, the previous inequality means
 $\bullet M^\omega < \bullet (M+1)(t_1-1)^{pn+m-1}{P'}$, i.e., $\bullet 0 M^\omega < \bullet 1(t_1-1)^{pn+m-1}P'$ or,
 equivalently, $\bullet M^\omega < 1 \bullet (t_1-1)^{pn+m-1} P'$, which is true thanks to the inequality~(\ref{ProofNonsimple1}), since $M = 2t_1-t_2-2$.

In order to demonstrate the inequality \eqref{ProofNonsimple3} for any index $i$ with $1 \leq i \leq pn+m$, we will show that
\begin{equation}\label{ProofNonsimple7}
g_{1}g_{2}g_{3}\cdots \prec_{lex} g_{i+1}g_{i+2}g_{i+3} \dots \quad \hbox{for}\ \ i=1,2,\ldots, pn+m\,.
\end{equation}
As $ g_0\bullet g_1 g_2 \cdots  g_{p\ell+pn+m}$ is the greedy representation, any
 of its suffixes satisfies the Parry condition as well, and thus the lexicographic
 ordering of the representations corresponds to the numerical ordering of the corresponding
 real numbers. Therefore, the inequality~ \eqref{ProofNonsimple7} together with
 validity of~\eqref{ProofNonsimple3} for $i=0$ implies validity of~\eqref{ProofNonsimple3} for all $i=0,1, \ldots, pn+m$.

The inequality~\eqref{ProofNonsimple7} can be equivalently rewritten to
$$(t_1-1-g_1) (t_1-1-g_2) (t_1-1-g_3) \cdots \succ_{lex} (t_1-1-g_{i+1}) (t_1-1-g_{i+2}) (t_1-1-g_{i+3}) \cdots$$
Looking into the last row of the table, it means
$$t_2 t_2 t_3 \cdots ({\rm Per})^n 0^{p\ell-m} \overline{1} (t_1-1-t_2) \cdots (t_1-1-t_m) \succ_{lex} \qquad\qquad\qquad$$
$$\qquad\qquad \succ_{lex} t_{i+1} t_{i+2} t_{i+3} \cdots ({\rm Per})^h 0^{p\ell-m} \overline{1} (t_1-1-t_2) \cdots (t_1-1-t_m)$$
for a certain power~$h$ of the period $Per$ in the range of $n \ge h \ge 0$.
The assumptions about the coefficients $t_1, t_2, \ldots, t_{m+p}$ guarantee that the last inequality is fulfilled.
\end{proof}
%%%%%%%%%%%%%%%%%%%%%%%%%%%%%%%%%%%%%%%%%%%%

\noindent{\bf B)} \quad For conversion of the string $y = \bullet (t_1-1)^{n+1}$, we use another auxiliary statement:

\smallskip

%%%%%%%%%%%%%%%%%%%%%%%%%%%%%%%%%%%%%%%%%%%%
\noindent\textbf{Statement 2}:\quad {\it Any other representation of the number
 $y = \bullet (t_1-1)^{n+1}$ in base~$\beta$ on alphabet~$\A$ begins with the prefix $ \bullet (t_1-1)^{n} \,.$}

\begin{proof}
We use the same arguments: since the string $(t_1-1)^{n+1}$ satisfies the
 Parry lexicographical condition, the representation $\bullet (t_1-1)^{n+1}$ is
  the greedy expansion of $y$ and $y<1$. According to Lemma \ref{coincidence}, we have to check that
$$ 0\bullet M^\omega < 1\bullet (t_1-1)^{i} \quad \hbox{for all } 1\leq i\leq n \,.$$
This follows from~\eqref{ProofNonsimple1}.
\end{proof}
%%%%%%%%%%%%%%%%%%%%%%%%%%%%%%%%%%%%%%%%%%%%

Now we can deduce the desired contradiction to the assumption of the existence of
 a $k$-block $p$-local function $\Phi$. Since Statement~2 holds for an arbitrary
  $n\in\N$, necessarily $\Phi((t_1-1)^{kp}) = (t_1-1)^k $. But, according to
   Statement~1, the string $\Phi((t_1-1)^{kp})$ has to be compounded  from blocks $H$ --- a contradiction.
\end{proof}
%%%%%%%%%%%%%%%%%%%%%%%%%%%%%%%%%%%%%%%%%%%%%%%%%%%%%%%%%%%

\bigskip

We illustrate on base~$\beta$, the larger root of the equation $X^2 = aX - b$,
 where $a,b \in \N, a\geq b+2, b\geq 1$ that our bound on the cardinality of
 alphabet in Theorem~\ref{nonSimple} is sharp. The R\'enyi expansion of unity is $d_{\beta}(1) = (a-1)(a-b-1)^\omega$.

We show that the smallest possible alphabet $\A = \{ 0, \ldots, a+b-2\}$ and
the smallest possible size $k=1$ of the block enable parallel addition by a $k$-block local function.

%%%%%%%%%%%%%%%%%%%%%%%%%%%%%%%%%%%%%%%%%%%%%%%%%%%%%%%%%%%
\begin{proposition}
Let $d_{\beta}(1) = (a-1)(a-b-1)^\omega$, where $a \ge b+2$, $b \ge 1$, be
the R\'enyi expansion of~$1$ in base~$\beta$. Parallel addition in base~$\beta$
 is possible on alphabet $\A=\{0, \ldots, a+b-2\}$, namely by means of a $1$-block local function.
\end{proposition}

By Proposition~18 in~\cite{FrPeSv2}, it is enough to show that the greatest digit
elimination from $\{0,\ldots, a+b-1\}$ to $\{0, \ldots, a+b-2\} = \A$ can be done in parallel:
%%%%%%%%%%%%%%%%%%%%%%%%%%%%%%%%%%%%%%%%%%%%%%%%%%%%%%%%%%%
\vskip0.2cm \hrule \vskip0.2cm

\noindent {\bf Algorithm~\textsl{GDE}($\beta^2=a\beta-b$)}: Base $\beta > 1$ satisfying
$\beta^2=a\beta-b$, with $a \ge b+2$, $b \ge 1$, parallel conversion (greatest digit elimination) from $\{0,\ldots, a+b-1\}$ to $\{0, \ldots, a+b-2\} = \A$.

\vskip0.2cm \hrule \vskip0.2cm

\noindent{\sl Input}: a finite sequence of digits $(z_j)$ from $\{0, \ldots, a+b-1\}$, with $z = \sum_j z_j\beta^j$.\\
{\sl Output}: a finite sequence of digits $(x_j)$ from $\{0, \ldots, a+b-2\}$, with $ z = \sum_j x_j \beta^{j}$.

\vskip0.2cm

\noindent\texttt{for each $j$ in parallel do}\\

1. \hspace*{0.3cm} \texttt{case}
    $\left\{\begin{array}{l}
        z_j = a+b-1\  \\
        a-1 \le z_j \le a+b-2\ \hbox{ \texttt{and}} \ \bigl( z_{j+1} \geq  a-1 \ \hbox{ \texttt{or}}\ z_{j-1}\geq a -1\Bigr) \ \\
        z_j = a-2 \ \hbox{ \texttt{and}} \ z_{j+1} = a+b-1 \ \hbox{ \texttt{and}} \ z_{j-1} = a+b-1\ \\
        z_j = a-2 \ \hbox{ \texttt{and}} \ z_{j+1} = a+b-1 \ \hbox{ \texttt{and}} \ z_{j-1} \ge a-1 \ \hbox{ \texttt{and}} \ z_{j-2} \geq a-1 \\
        z_j = a-2 \ \hbox{ \texttt{and}} \ z_{j-1} = a+b-1 \ \hbox{ \texttt{and}} \ z_{j+1}\ge a-1 \ \hbox{ \texttt{and}} \ z_{j+2} \geq a-1 \\
        z_j = a-2 \ \hbox{ \texttt{and}} \ z_{j \pm 1} \ge  a-1 \ \hbox{ \texttt{and}} \ z_{j \pm 2} \geq a-1 \ \\
    \end{array} \right\} $\\

\hspace*{1.5cm} \texttt{then} $q_j:=1$\\[1mm]
\hspace*{1.5cm} \texttt{else} $q_j:=0$\\

2. \hspace*{0.3cm} $x_j := z_j - a q_j + b q_{j+1} + q_{j-1}$

\vskip0.2cm \hrule \vskip0.2cm
%%%%%%%%%%%%%%%%%%%%%%%%%%%%%%%%%%%%%%%%%%%%%%%%%%%%%%%%%%%

\begin{proof}
Let us denote  $w_j := z_j- a q_j$; and remind that $q_j \in \{0, 1\}$ for any $j$, and thus $b q_{j+1} + q_{j-1} \in \{0, 1, b, b+1\}$.
\begin{itemize}
    \item If $z_j \in \{ 0,\ldots, a-3\}$, then $x_j = z_j + bq_{j+1} + q_{j-1} \in \{ 0,\ldots, a+b-2\} = \A$.

    \item If $z_j = a+b-1$, then $w_j=b-1$, thus $0 \le x_j \le 2b \le a+b-2 \in \A$, since $a \ge b+2$.

    \item When $a-1 \le z_j \le a+b-2$, and $z_{j - 1} \geq a-1$ or $z_{j + 1} \geq a-1$, then $-1 \le w_j \le b-2$ and $q_{j+1} + q_{j-1} \in \{1,2\}$. Thus $x_j\in \{0, \ldots, 2b-1\} \subset \A$.

    \item When $a-1 \le z_j \le a+b-2$ and both its neighbours $z_{j \pm 1} < a-1$, then $w_j =  z_j$ and $q_{j+1} = q_{j-1} = 0$. Thus $x_j \in \A$.

    \item If $z_j = a-2$ and $q_j =1$, then necessarily $q_{j \pm 1} = 1$. Since $w_j =-2 $, we get $ x_j =b-1\in \A$.

    \item If $z_j = a-2$ and $q_j =0$, then $w_j = a-2$, and $q_{j-1} + q_{j+1} \in \{0, 1\}$.
    Therefore, the resulting $x_j \in \{a-2, a-1, a+b-2\} \subset \A$.
\end{itemize}
Lastly, it is obvious that a string of zeroes is not converted into a string of
non-zeroes by this algorithm, so all the necessary conditions of parallel addition are fulfilled.
\end{proof}
%%%%%%%%%%%%%%%%%%%%%%%%%%%%%%%%%%%%%%%%%%%%%%%%%%%%%%%%%%%

%%%%%%%%%%%%%%%%%%%%%%%%%%%%%%%%%%%%%%%%%%%%%%%%%%%%%%%%%%%
\begin{proposition}
Let $\beta$ satisfy $\beta^2 = a\beta -b$, with $a \ge b+2$, $b \ge 1$. Parallel addition
in base~$\beta$ is possible on any alphabet of cardinality $a+b-1$ of the form $\A = \{-d, \ldots, a+b-2-d\}$ for $b \le d \le a-2$.
\end{proposition}

\begin{proof}
Every letter $h$, $0 \le h \le a-2$, is fixed by the above algorithm. So for $b \le d \le a-2$,
both letters $d$ and $a+b-2-d$ are fixed, and, by Corollary~24 in~\cite{FrPeSv2}, parallel
addition is possible on any alphabet of the form $\A = \{-d, \ldots, a+b-2-d\}$ with $b \le d \le a-2$.
\end{proof}
%%%%%%%%%%%%%%%%%%%%%%%%%%%%%%%%%%%%%%%%%%%%%%%%%%%%%%%%%%%

It is an open question to prove that in base~$\beta$ satisfying $\beta^2 = a\beta -b$,
with $a \ge b+2$, $b \ge 2$, parallel addition is not possible on alphabets of positive and
negative contiguous integer digits not containing $\{-b, \ldots,0, \ldots,b\}$, as it is the
case in rational base $\beta = a/b$ when $b \ge 2$, see \cite{FrPeSv2}.

%%%%%%%%%%%%%%%%%%%%%%%%%%%%%%%%%%%%%%%%%%%%%%%%%%%%%%%%%%%%%%%%%%%%%%%%%%%%%%%%%%%%%%%%%
%%%%%%%%%%%%%%%%%%%%%%%%%%%%%%%%%%%%%%%%%%%%%%%%%%%%%%%%%%%%%%%%%%%%%%%%%%%%%%%%%%%%%%%%%

\section{Upper bounds on minimal alphabet allowing block parallel addition}\label{kblockhelps}

%%%%%%%%%%%%%%%%%%%%%%%%%%%%%%%%%%%%%%%%%%%%%%%%%%%%%%%%%%%%%%%%%%%%%%%%%%%%%%%%%%%%%%%%%

%%%%%%%%%%%%%%%%%%%%%%%%%%%%%%%%%%%%%%%%%%%%%%%%%%%%%%%%%%%
\begin{thm}\label{Pavel}
Given a base~$\beta$ and an alphabet~$\B$ of contiguous integers containing~$0$; let us
suppose that there exist non-negative integers $\ell$ and $s$ such that for
any $x = x_{n} \cdots x_0\bullet$ and $y = y_{n} \cdots y_0\bullet$ from
${\rm Fin}_\B(\beta )$ the sum $x+y$ has a $(\beta,  \B)$-representation of the form
$$z= x+y =  z_{n+\ell} \cdots z_0\bullet z_{-1} \cdots z_{-s}\,.$$
Then there exists a $k$-block $3$-local function performing parallel addition
in base~$\beta$ on the alphabet $\A = \B+ \B$, where $k= 2(\ell+s)$.
\end{thm}

\begin{proof}
According to the assumptions, any $x= \sum_{j=0}^{k-1}  x_j \beta^j$ with $x_j \in \B+ \B$
can be written as $x= \sum_{j=-s}^{k+\ell-1}  x'_j \beta^j$  with $x'_j \in \B$. And
thus any $z= \sum_{j=0}^{k-1}  z_j \beta^j$ with $z_j \in \A+ \A$ can be written as
$$z= \sum_{j=-2s}^{k+2\ell-1} z'_j \beta^j \quad \hbox{  with } \ z'_j \in \B\,.$$
It means that for any $u\in (\A + \A)_{(k)}$ there exist
$$ L(u) \in \B_{(2\ell)}, \quad C(u) \in \B_{(k)},  \ \ \hbox{and} \ \ S(u) \in \B_{(2s)}$$
such that
\begin{equation}\label{3block}
u=L(u)\beta^k + C(u)+  S(u)\beta^{-2s}\, .
\end{equation}
It may happen that for $u\in (\A+ \A)_{(k)}$ there exist several triples $L(u)$, $C(u)$, $S(u)$
with the required property. But for any $u$, we fix  just one triple. We can set
\begin{equation}\label{choice}
L(u)= S(u)=0 \quad \hbox{and} \quad C(u) = u \quad  \hbox{for any \ } u\in   \B_{(k)}.
\end{equation}
In particular,  we put $L(0) =C(0)=S(0)=0$.\\
Let us  define a $3$-local function $\Phi$ with domain $(\A_{(k)} + \A_{(k)})^3$ by
\begin{equation}\label{3block1}
\Phi(f, g, h) = L(h) + C(g)+   S(f) \beta^{2\ell}\,.
\end{equation}
As $k=2(\ell+s)$, $\B_{(k)} = \B_{(2\ell)} +\B_{(2s)}\beta^{2\ell} $, and the function
$\Phi$ maps $(\A_{(k)} + \A_{(k)})^3$ to $\B_{(k)} + \B_{(k)} = \A_{(k)}$.

Let $\cdots u_{2}u_1u_0u_{-1}u_{-2} \cdots $ be a sequence with finitely many non-zero $u_j \in \A_{(k)} + \A_{(k)}$.
We show that
$$ \sum_{j\in \Z}u_j \beta^{jk} =  \sum_{j\in \Z}v_j \beta^{jk}\,, \ \hbox{where}\ v_j =  \Phi(u_{j+1}\,u_{j}\, u_{j-1}).$$
Indeed, by~\eqref{3block} and~\eqref{3block1}, we have
$$\sum\limits_{j\in \Z}u_j \beta^{jk} = \sum\limits_{j\in\Z}L(u_j)\beta^{k(j+1)} + \sum\limits_{j\in\Z} C(u_j)\beta^{kj} + \sum\limits_{j\in\Z}S(u_j)\beta^{kj-2s}=$$

$$ = \sum\limits_{j\in \Z}L(u_{j-1})\beta^{kj} + \sum\limits_{j\in \Z} C(u_j)\beta^{kj} + \beta^{2\ell}\sum\limits_{j\in \Z}S(u_{j+1})\beta^{kj} = \sum\limits_{j\in\Z}\Phi\bigl(u_{j+1}\,u_j\,u_{j-1}\bigr)\beta^{kj}.
$$
Our choice $L(0)=C(0)=S(0)=0$ guarantees that  the sequence \ $\cdots v_{2}v_1v_0v_{-1}v_{-2} \cdots $
\ has only finitely many non-zero elements as well. Therefore, $\Phi$ is the desired $k$-block
 $3$-local function performing parallel addition in base~$\beta$ on the alphabet $\A = \B + \B$.
\end{proof}
%%%%%%%%%%%%%%%%%%%%%%%%%%%%%%%%%%%%%%%%%%%%%%%%%%%%%%%%%%%

\begin{remark}\label{fixation}
From equations~\eqref{choice} and~\eqref{3block1} in the previous proof we see
that $\Phi(u, u, u) = u$ for any $u\in\B_{(k)}$. It means that the infinite constant
sequence $(u)_{j\in\Z}$ is fixed by the corresponding parallel algorithm for any $u \in \B_{(k)}$.
\end{remark}

%%%%%%%%%%%%%%%%%%%%%%%%%%%%%%%%%%%%%%%%%%%%%%%%%%%%%%%%%%%
\begin{proposition}\label{with(PF)}
Let $\beta >1$ be a number with the (PF) Property. Then there exists $k\in\N$ such that
$k$-block parallel addition in base~$\beta$ is possible on the alphabet
 $\A = \{0, 1, \ldots, 2\lfloor \beta \rfloor\}$, and also on the alphabet
  $\A = \{-\lfloor \beta \rfloor, \ldots, -1, 0, 1, \ldots, \lfloor \beta \rfloor\}$.
\end{proposition}

\begin{proof}
Let $d_\beta(1) = t_1t_2\cdots$ be the R\'enyi expansion of unity in base~$\beta$;
obviously, $t_1= \lfloor \beta \rfloor$. We apply the previous Theorem~\ref{Pavel}
to $\B = \{0, 1, \ldots, \lfloor \beta \rfloor\}$. In \cite{BuGaFrKr}, the numbers
$x$ for which the greedy expansion in base~$\beta$ has a form $x_nx_{n-1}\cdots x_1x_0\bullet$
were called $\beta$-integers. The set of $\beta$-integers is usually denoted $\Z_\beta$.
Using the Parry lexicographical condition, we can write formally
$$\Z_\beta = \Bigl\{ \sum_{j=0}^n x_j\beta^j \mid x_j \in  \B\ \  \hbox{and } \ \
 x_jx_{j-1}\cdots x_1x_0 \prec t_1t_2t_3\cdots \ \ \hbox{for any} \ j=0,1,\ldots, n \Bigr\}\, .$$

Let us denote by
$$\B[\beta]= \Bigl\{ \sum_{j=0}^n x_j\beta^j \mid x_j \in  \B \Bigr\}$$
Clearly, $\Z_\beta \subset  \B[\beta]$, but, in general, the opposite inclusion does
not hold. Nevertheless, for a given base~$\beta$ with the (PF) Property, there exists
 a constant $h\in\N$ such that any $x \in \B[\beta]$ can be written as a sum of at
 most $h$ elements from $\Z_\beta$:
\begin{itemize}
    \item If $t_1 >1$, then $h=2$, since any coefficient $x_j \in \B$ can be written
    as $x_j= x_j'+ x_j''$, where $x_j', x_j'' < t_1$. Thus $\sum_{j=0}^n x_j\beta^j
    = \sum_{j=0}^n x_j'\beta^j+ \sum_{j=0}^n x_j''\beta^j$ and coefficients in both sums on the right side satisfy the Parry condition.
    \item If $t_1 =1$, we can take as $h$ the minimal integer $h\geq 2$ such
    that $t_{h}\neq 0$. This choice of $h$ guarantees that $d_\beta(1) = t_10^{h-2}t_{h}\cdots$
    and that any representation $z_nz_{n-1}\cdots z_1z_0\bullet z_{-1} z_{-2} \cdots$ of a number
     $z$ in which each nonzero coefficient $ z_j = 1$ is followed by $h-1$ zeros $ z_{j-1} = z_{j-2} = \cdots = z_{j-h +1} = 0$,
     is already the greedy expansion of $z$.
        Therefore, any $x =\sum_{j=0}^n x_j\beta^j \in \B[\beta]$ can be
        written as $x=x^{(0)} +  x^{(1)} + \cdots + x^{(h-1)}$, with $x^{(c)} = \sum_{j=0}^n x_j^{(c)}\beta^j \in \Z_{\beta}$ defined by
        $$ x_j^{(c)} = \left\{ \begin{array}{cl}
            0 & \hbox{if\ } j\neq c \mod h\\
            x_j & \hbox{if\ } j = c \mod h \,.
        \end{array}\right. $$
\end{itemize}

Bernat studies in \cite{Bernat} the number of fractional digits in the greedy expansion
of $x+y$ of two $\beta$-integers $x$ and $y$. He shows that if $\beta$ is a Perron
 number (i.e., an algebraic integer with all its algebraic conjugates of modulus strictly
 less than $\beta$) with no algebraic conjugate of modulus~$1$, then there exists a
  constant $L_{\oplus}\in\N$, such that if $x+y$ has finite greedy $\beta$-expansion,
  then the number of fractional digits in the greedy expansion of $x+y$ is less than or
  equal to $L_{\oplus}$. Let us stress that the value $L_{\oplus}$ is effectively computable
  when $\beta$ is a Parry number. Since our base~$\beta$ has the (PF) Property, the greedy
   expansion of the sum of any two $\beta$-integers is finite, and thus we can apply the
   previous Theorem~\ref{Pavel} with $s = h L_{\oplus}$.

In order to exploit the Theorem~\ref{Pavel}, we have to find also a suitable $\ell$.
Let $\ell$ be the smallest integer such that $\frac{2\lfloor\beta \rfloor }{\beta-1} < \beta^\ell$.
Since for any $x \in \B[\beta]$ we have $x = x_{n} \cdots x_0\bullet \leq \lfloor\beta \rfloor \frac{\beta^{n+1}-1}{\beta-1}$,
we can estimate $x+y = x_{n} \cdots x_0\bullet \ + \ y_{n} \cdots y_0\bullet \leq 2\lfloor\beta \rfloor \frac{\beta^{n+1}}{\beta-1} < \beta^{n+\ell+1}$.
The inequality $z= x+y <  \beta^{n+\ell+1}$ implies that at least one representation of $z$ (namely the greedy expansion prolonged
to the left by zero coefficients if needed) has the form $z = z_{n+\ell} \cdots z_0\bullet z_{-1} z_{-2}\cdots$.

Using Theorem~\ref{Pavel}, we have proved that parallel addition is possible on the alphabet
${\A} = \{0,1,\ldots, 2 \lfloor \beta \rfloor\}$. According to Remark~\ref{fixation}, the
 sequence $(h)_{j\in \Z}$ is fixed by the algorithm for parallel addition for any $h \in \{ 0,1, \ldots, \lfloor \beta
\rfloor \} = \B$. Therefore, due to Corollary~24 in~\cite{FrPeSv2}, the alphabet
${\A} - \lfloor \beta \rfloor = \{-\lfloor \beta \rfloor, \ldots, 0, \ldots, \lfloor \beta \rfloor\}$ allows parallel addition as well.
\end{proof}
%%%%%%%%%%%%%%%%%%%%%%%%%%%%%%%%%%%%%%%%%%%%%%%%%%%%%%%%%%%

Combining Proposition~\ref{with(PF)}, Theorem~\ref{nonSimple}, and Theorem~\ref{Simple}, we can derive the following conclusions:

\begin{corollary}\label{ConclusionSimple}
Let $d_\beta(1) = t_1t_2\cdots t_m$, with $t_1 \ge t_2 \ge \cdots \ge t_m \ge 1$ be
the R\'enyi expansion of~$1$ in base~$\beta$. Then there exists $M \in \N$ such that
parallel addition by a $k$-block local function in a non-integer base~$\beta$ is possible
on the alphabet $\A = \{0,1,\ldots, M\}$ with $t_1 + t_m \leq M \leq 2t_1$.
\end{corollary}

\begin{corollary} \label{ConclusionNonSimple}
Let $d_\beta(1) = t_1t_2\cdots t_mt^\omega$ with $t_1 > t_2\geq t_2\geq \cdots \geq  t_m > t\geq 1$
be the R\'enyi expansion of~$1$ in base~$\beta$.Then there exists $M \in \N$ such that parallel
addition by a $k$-block local function in base~$\beta$ is possible on the alphabet $\A = \{0,1,\ldots, M\}$ with $2t_1 - t_2 - 1 \leq M \leq 2t_1$.
\end{corollary}

On those bases $\beta$ that are $d$-bonacci numbers we will demonstrate how the
concept of $k$-block local function can substantially reduce the cardinality of alphabet which allows parallel addition:

%%%%%%%%%%%%%%%%%%%%%%%%%%%%%%%%%%%%%%%%%%%%%%%%%%%%%%%%%%%
\begin{corollary}\label{DBonacci} Let $\beta$ be  a $d$-bonacci number for some $d\in\N$, $d \ge 2$.
\begin{itemize}
    \item If an alphabet~$\A$ allows $1$-block parallel addition in base~$\beta$,
    then its cardinality is $\#\A \geq d+1$.
    \item There exists $k\in\N$ such that $k$-block parallel addition in base~$\beta$ is
    possible on the alphabets $\A = \{0,1,2\}$ and $\A = \{-1,0,1\}$, and these alphabets cannot be further reduced.
\end{itemize}
\end{corollary}

\begin{proof}
The minimal polynomial of a $d$-bonacci number is $f(X) = X^d - X^{d-1} - X^{d-2} - \cdots - X - 1$.
Theorem~\ref{zdola} says that $1$-block parallel addition is possible only on an alphabet with cardinality at least $|f(1)| +2 = d+1$.

The R\'enyi expansion of unity for a $d$-bonacci number is $d_\beta(1) = 1^d$, and
thus the $d$-bonacci number satisfies the (PF) Property. Since $\lfloor\beta\rfloor = 1$,
due to Proposition~\ref{with(PF)}, $k$-block parallel addition in base~$\beta$ is possible on
the alphabets $\A = \{0,1,2\}$ and $\A = \{-1,0,1\}$. With respect to Corollary~\ref{Dbonacci}, this alphabet is minimal.
\end{proof}
%%%%%%%%%%%%%%%%%%%%%%%%%%%%%%%%%%%%%%%%%%%%%%%%%%%%%%%%%%%

\begin{example}\label{Tribonacci}
In \cite{Bernat2}, Bernat computes the value of $L_{\oplus}$ for the Tribonacci base,
namely $L_{\oplus} = 5$. So the parameter $s$ in Theorem~\ref{Pavel} is equal to $5$. It is
easy to see that $\ell = 2$. Thus, addition in the Tribonacci base is $14$-block $3$-local
parallel on the alphabets $\A = \{0,1,2\}$ or $\A = \{-1,0,1\}$.
\end{example}

\begin{remark}\label{CNS}
This article deals mainly with positive bases~$\beta$. However,
Theorem~\ref{Pavel} can be applied to complex bases as well.
One such class of bases defines the so-called {\em Canonical Number
Systems} (CNS), see \cite{Kovacs} and \cite{Pethoe}.

An algebraic number~$\beta$ and the alphabet $\B = \{0, 1, \ldots, | N(\beta) |-1\}$,
where $N(\beta)$ denotes the norm of $\beta$ over $\Q$, form a Canonical Number System, if any
 element~$x$ of the ring of integers $\Z[\beta]$ has a unique representation in the
 form $x = \sum_{k=0}^n x_k \beta^k$, where $x_k \in \B$ and $x_n \neq 0$.

In particular, it means that the sum of two elements of
$\Z[\beta]$ has also a finite representation in the
 form $\sum_{k=0}^m x_k \beta^k$, where $x_k \in \B$ and $x_m \neq 0$, and thus in
Theorem~\ref{Pavel} we can set $s = 0$. It can be proved that CNS
guarantees also the existence of the constant $\ell$ required in
that theorem. We can conclude that, in CNS, block parallel
addition is possible on the alphabet $\A = \{0, 1, \ldots, 2
|N(\beta)|-2\}$ or in the alphabet $\A = \{-|N(\beta)|+1, \ldots, 0,
\ldots, |N(\beta)|-1\}$.

More specifically for the Penney numeration system, the base $\beta =
\imath-1$ has norm $N(\beta) = 2$, and together with the
alphabet $\B = \{0, 1\}$ forms a CNS. Therefore, due to
Theorem~\ref{Pavel}, block parallel addition in the Penney
numeration system is possible not only on the alphabet $\A = \{-1,
0, 1\}$ (as shown by Herreros), but also on alphabet $\A = \{0, 1,
2\}$.

\end{remark}

%%%%%%%%%%%%%%%%%%%%%%%%%%%%%%%%%%%%%%%%%%%%%%%%%%%%%%%%%%%%%%%%%%%%%%%%%%%%%%%%%%%%%%%%%
%%%%%%%%%%%%%%%%%%%%%%%%%%%%%%%%%%%%%%%%%%%%%%%%%%%%%%%%%%%%%%%%%%%%%%%%%%%%%%%%%%%%%%%%%

\section{Comments and open questions}

%%%%%%%%%%%%%%%%%%%%%%%%%%%%%%%%%%%%%%%%%%%%%%%%%%%%%%%%%%%%%%%%%%%%%%%%%%%%%%%%%%%%%%%%%

When designing the algorithms for parallel addition in a given
base~$\beta$,  we need to take into consideration three core
parameters:

1) the cardinality $\# \A$ of the used alphabet~$\A$,

2) the width~$p$ of the sliding window, i.e., the number~$p$
appearing in the  definition of the $p$-local function $\Phi$, and

3) the length~$k$ of the blocks in which we group the digits of the $(\beta, \A)$-representations for $k$-block parallel addition.

\noindent There are mathematical reasons (for example comparison
of numbers) and even  more technical reasons to minimize all these
three parameters. But intuitively, the smaller is one of the
parameters, the bigger have to be the other ones. The question
which relationship binds the values $\# \A$, $p$, and~$k$ is far
from being answered.

\bigskip

In that respect, we are able to list just several isolated observations made for specific bases:

\begin{itemize}
    \item In \cite{FrPeSv1}, we studied $1$-block parallel addition, i.e., $k$ was
     fixed to~$1$. For base~$\beta$ being the Fibonacci number (i.e. the golden mean
      $\frac{1+\sqrt{5}}{2}$), we gave a parallel algorithm for addition on the alphabet
       $\A = \{-3, \ldots, 0,\ldots, 3\}$ by a $13$-local function. On the other hand, for
       the same base, we have also described an algorithm for parallel addition on the minimal
       alphabet $\A = \{-1, 0,1\}$, where the corresponding function $\Phi$ is $21$-local.

    \item The $d$-bonacci bases illustrate that if we do not care about the length~$k$ of the
     blocks, the alphabet can be substantially reduced, namely to $\A = \{0,1,2\}$, see
     Corollary~\ref{DBonacci}. But the price for that is rather high; already for the Tribonacci
      base our algorithm requires blocks of length $k=14$, see Example~\ref{Tribonacci}.

    \item If we fix in the Penney numeration system the value $k=1$, an alphabet of cardinality~$5$ is
    necessary for parallel addition. Herreros in~\cite{Herreros} provided an algorithm for parallel
    addition in the Penney base $\beta = \imath - 1$ on the alphabet $\A = \{-1,0,1\}$, but his
    algorithm uses $k=4$. This value is not optimal; we have found (not yet published)
    that $k=2$ is enough to perform parallel addition on the alphabet $\A = \{-1, 0, 1\}$.
\end{itemize}

\bigskip

Besides the width~$p$ of the sliding window as such, there is
another characteristic  which is desired for the algorithms
performing parallel addition, namely to be {\em neighbour-free}.
This property has to do with the way how one decides within the
first step of the algorithm what value $q_j$ to choose at the
$j$-th position of the processed string; which is in fact the key
task of the algorithm, as otherwise, once having the correct set
of the values $q_j$ after the first step, one only deducts the
$q_j$-multiple of an appropriate form of a representation of zero,
and the task is finished. Being neighbour-free means that the
value $q_j$ depends only on the digit on the $j$-th position of
the processed string, irrespective of its neighbours. Note that
this is something else than being $1$-local! On the other hand, an
algorithm of parallel addition which is not neighbour-free, is
called {\em neighbour-sensitive}, see the discussion in~\cite{FrPeSv1}.

For integer bases, as explained in Remark~\ref{int-base}, the
concept of $k$-block  parallel addition with $k\ge 2$ is not
interesting from the point of view of the minimality of the
cardinality of the alphabet. However, grouping of digits into
$k$-blocks can improve the parallel algorithm in another way,
namely with respect to the neighbour-free property.

For instance, in base $\beta = 2$, there is $1$-block parallel
addition doable on the minimal alphabet $\A = \{-1,0,1\}$ by the
neighbour-sensitive algorithm of Chow and
Robertson~\cite{ChowRobertson}. But $2$-block addition here means
just addition in base $\beta^2= 4$ on alphabet $\A_{(2)} = \{-3,
\ldots, 0, \ldots, 3\}$, and is performable by the simpler
algorithm of Avizienis~\cite{Avizienis}, which is neighbour-free.

\bigskip

The most common reason why to work in a numeration system  with an
algebraic base~$\beta$, instead of a system with base~$2$ or~$10$,
consists in the requirement to perform precise computations in the
algebraic field $\mathbb{Q}(\beta)$. If the base~$\beta$ is not
'nice enough', we can choose another base $\gamma$ such that
$\mathbb{Q}(\beta) = \mathbb{Q}(\gamma)$ and then work in the
numeration system with the base $\gamma$. The question is which
base in $\mathbb{Q}(\beta)$ is 'nice enough' and how to find it
effectively.

\begin{itemize}
    \item  Certainly, the 'beauty' of the Pisot bases is not
    questionable. Q.~Cheng and J.~Zhu in\cite{China} described an algorithm
    for finding a Pisot number which generates the whole algebraic field $\mathbb{Q}(\gamma)$.
    \item From another point of view, a base allowing parallel addition
    on a binary alphabet would be 'beautiful' as well; but there is no example of such a base known yet. May it exist?
\end{itemize}

%%%%%%%%%%%%%%%%%%%%%%%%%%%%%%%%%%%%%%%%%%%%%%%%%%%%%%%%%%%%%%%%%%%%%%%%%%%%%%%%%%%%%%%%%
%%%%%%%%%%%%%%%%%%%%%%%%%%%%%%%%%%%%%%%%%%%%%%%%%%%%%%%%%%%%%%%%%%%%%%%%%%%%%%%%%%%%%%%%%

\section*{Acknowledgements}

%%%%%%%%%%%%%%%%%%%%%%%%%%%%%%%%%%%%%%%%%%%%%%%%%%%%%%%%%%%%%%%%%%%%%%%%%%%%%%%%%%%%%%%%%

The second author acknowledges financial support by   the Grant
Agency of the Czech Technical University in Prague, grant
SGS11/162/OHK4/3T/14.  The third and fourth authors acknowledge
financial support by the Czech Science Foundation grant 13-03538S.

%%%%%%%%%%%%%%%%%%%%%%%%%%%%%%%%%%%%%%%%%%%%%%%%%%%%%%%%%%%%%%%%%%%%%%%%%%%%%%%%%%%%%%%%%
%%%%%%%%%%%%%%%%%%%%%%%%%%%%%%%%%%%%%%%%%%%%%%%%%%%%%%%%%%%%%%%%%%%%%%%%%%%%%%%%%%%%%%%%%

%%%%%%%%%%%%%%%%%%%%%%%%%%%%%%%%%%%%%%%%%%%%%%%%%%%%%%%%%%%%%%%%%%%%%%%%%%%%%%%%%%%%%%%%%
%%%%%%%%%%%%%%%%%%%%%%%%%%%%%%%%%%%%%%%%%%%%%%%%%%%%%%%%%%%%%%%%%%%%%%%%%%%%%%%%%%%%%%%%%

\end{document}